\documentclass{article}

\usepackage{amsmath}
\usepackage{amssymb}
\usepackage{amsthm}
\usepackage{latexsym}
\usepackage{bbold}
\usepackage{stmaryrd}

\newtheorem{theorem}{Theorem}

\newtheorem{lemma}[theorem]{Lemma}
\newtheorem{definition}[theorem]{Definition}

\theoremstyle{definition}
\newtheorem{example}[theorem]{Example}

\newenvironment{romanenumerate}{\begin{enumerate}}
        {\end{enumerate}} 

\newcommand{\sst}{\; | \;}  
\newcommand{\integers}{\ensuremath{\mathbb{Z}}} 
\newcommand{\naturals}{\ensuremath{\mathbb{N}}} 
\newcommand{\cs}{,\;} 
\newcommand{\qs}{\:} 
\newcommand{\A}{\mathbb{A}} 
\newcommand{\B}{\mathbb{B}} 
\newcommand{\F}{\ensuremath{\mathcal{F}}} 
\newcommand{\LyndonGrp}{\ensuremath{F^{\integers [t]}}} 
\newcommand{\Aut}{\mathrm{Aut}} 
\newcommand{\sgn}{\mathrm{sgn}} 
\renewcommand{\t}[2]{t_{{#1},{#2}}} 
\newcommand{\braced}[4]{\left\{\begin{array}{ll} {#1} & {#2} \\ {#3} & {#4} \end{array}\right.}

\title{Compressed words and automorphisms in fully residually free groups}
\author{Jeremy Macdonald\\Department of
  Mathematics and Statistics, McGill University\\
  805 Sherbrooke Street West \\
Montr\'eal, Qu\'ebec,
Canada,
H3A 2K6 \\
  email:
  \texttt{jmacdonald@math.mcgill.ca}}
\begin{document}
\maketitle{}

\begin{abstract}
We show that the compressed word problem in a finitely generated fully residually free group ($\F$-group)
is decidable in polynomial time, and use
this result to show that the word problem in the automorphism group of an $\F$-group is decidable in
polynomial time.
\end{abstract}
\tableofcontents
\section{Preliminaries}
The \emph{word problem} for a finitely presented group $G=\langle X\,|\, R\rangle$ asks, given a word
$w$ over the alphabet $X^{\pm}=X\cup X^{-1}$, whether $w$ represents the identity element of $G$.  Being proposed for study by
Dehn in 1911, decidabliliy of the word problem for particular groups and classes of groups
was the main focus of study,
without regard to the efficiency of the proposed algorithms.  Once computational complexity became
of interest, time complexity of word problems was considered and has now been studied in many classes of groups.
One such class was the
automorphism group of a finite rank free group.  The problem reduces, with an exponential increase in size, to the
word problem in the underlying free group.  Schleimer has shown (\cite{Schleimer08}) that one can encode
the exponential expansion using Plandowski's techinque of \emph{compressed words} and,
using an algorithm for comparing compressed words (\cite{Plandowski94}), obtain a polynomial time algorithm.  We use a
similar strategy to obtain a polynomial time algorithm for the word problem in the automorphism group
of a finitely generated fully residually free group.

\subsection{The compressed word problem}
A \emph{straight-line program} (SLP) is a tuple
$\A=(X,\mathcal{A},A_n,\mathcal{P})$ consisting of a finite alphabet $\mathcal{A}=\{A_n,\ldots,A_1\}$
of \emph{non-terminal symbols}, a finite alphabet $X$ of \emph{terminal symbols},
a \emph{root} non-terminal $A_n\in\mathcal{A}$, and a set of \emph{productions}
$\mathcal{P}=\{A_i\rightarrow W_i\sst 1\leq i \leq n\}$ where $W_i\in\{A_j A_k \sst  j,k < i\}\cup X\cup\{\phi\}$,
where $\phi$ represents the empty word.  Computer scientists will recognize SLPs as a type of context-free grammar.
We `run' the program $\A$
by starting with the word $A_n$ and replacing each non-terminal $A_i$ by $W_i$ and continuing this replacement
procedure
until only terminal symbols remain.  The condition $j,k<i$ ensures that the program terminates.
The resulting word is denoted $w_{\A}$, and
we also denote by $w_{A_i}$ the result of running the same program starting with $A_i$ instead of the root $A_n$.
The SLP $\A$ (and, abusing language, $w_{\A}$) is also called a \emph{compressed word} over $X$.
The reader may consult \cite{Schleimer08} for a more detailed introduction to compressed words.

The \emph{production tree} associated with a non-terminal $A_m$
is the rooted binary tree with root labelled $A_m$ and where vertex
$A_i$ has children as follows: if $A_i\rightarrow A_j A_k$ then $A_i$ has left child $A_j$ and right
child $A_k$, if $A_i\rightarrow x$ (where $x\in X$) then $A_i$ has a single child labelled $x$, and
if $A_i\rightarrow \phi$ then $A_i$ has a single child labelled by the empty word $\phi$.
Notice that $w_{A_m}$ is the word appearing
at the leaves of the production tree.  We say that $A_m$ \emph{produces} $w_{A_m}$.

Let the size $|\A|$ of an SLP be the number $n$ of non-terminal symbols.  Note that the number of
bits required to write down $\A$ is $O(n\log_2 n)$ (the factor of $\log_2 n$ appears in writing down the
non-terminal symbols $A_i$).  An SLP with $n$ non-terminal symbols can encode a word $w_A$ of
length $2^n$.  Any algorithm that takes as input a  word over the alphabet $X$ can, of course, be used on compressed
words over $X$ by simply running the algorithm on $w_{\A}$, but this converts a time $f(n)$ algorithm to
one that runs in time $O(f(2^{|\A|}))$.  The goal then is to develop algorithms that work directly with the
SLP without expanding it.

In this paper we consider the \emph{compressed word problem} for finitely generated
fully residually free groups.
For an alphabet $X$, let $X^{-1}$ be the set of symbols $\{x^{-1}\sst x\in X\}$ and
set $X^{\pm}=X\cup X^{-1}$.
If $G$ is a group presented by $G=\langle X\sst R\rangle$ the compressed word problem
asks to decide, given a compressed word $\A$ over $X^{\pm}$, whether $w_{\A}$ represents the identity
element of $G$.

We will use the following result of Lohrey \cite{Lohrey04} that solves the compressed word problem
for free groups in polynomial time:
\begin{lemma}[Lohrey] \label{Lem:CWPinF}
There is a polynomial time algorithm which, given a straight-line program $\A$ over the
alphabet $X^{\pm}$, decides whether $w_{\A}=1$ in the free group on $X$.
\end{lemma}
Lohrey's result relies on the fundamental result of Plandowski \cite{Plandowski94}:
\begin{lemma}[Plandowski's Algorithm] There is a polynomial time algorthim which, given straight-line programs
$\A$ and $\B$ over an alphabet $X$, decides if $w_A=w_B$ (as words in the free monoid over $X$).
\end{lemma}
A nice description of both results and their proofs is given in \cite{Schleimer08}.

\subsection{Fully residually free groups and Lyndon's group $\LyndonGrp$}
\begin{definition}
A group $G$ is
\emph{fully residually free} if for every finite set $\{g_1,g_2,\ldots,g_n\}$ of elements of $G$ there exists a free group $F$
and a homomorphism $\varphi: G\rightarrow F$ such that $\varphi(g_i)\neq 1$ for all $i=1,2,\ldots,n$.
We refer to finitely generated fully residually free groups as $\F$-groups (they are also known as
\emph{limit groups}).
\end{definition}
Finitely generated free groups are $\F$-groups, and
the first example of a non-free $\F$-group was \emph{Lyndon's group} $\LyndonGrp$,
introduced in \cite{Lyndon60}.
$\F$-groups are now known to be
precisely the finitely-generated subgroups of $\LyndonGrp$
(\cite{IrredAffineII}).
We will use a description of $\LyndonGrp$ in terms of HNN-extensions, following
\cite{RegularFreeLengthFunctions} rather than \cite{Lyndon60}.  The construction is as follows.

For a group $G$, let $R(G) $ be a set of representatives of conjugacy classes of generators of
all proper cyclic centralizers of $G$.
That is, every centralizer in $G$ which is cyclic is conjugate to $C_G(u)=\langle u \rangle$ for some $u\in R(G)$, and
no two elements of $R(G)$ are conjugate. Then the \emph{extension of (all) cyclic centralizers of $G$} is the
HNN-extension
\begin{eqnarray}\label{Eqn:ExtOfAllCent}
\langle G,\t{u}{i} \; (u\in R(G),i\in \naturals) \sst
	\forall\qs (u\in R(G),i,j\in\naturals) \qs [\t{u}{i},u]=[\t{u}{i},\t{u}{j}]=1\;  \rangle.
\end{eqnarray}
Let $F$ be a  free group.  Then Lyndon's group $\LyndonGrp$ is (isomorphic to) the direct limit (i.e. union) of the infinite
chain of groups
\begin{eqnarray} \label{Eqn:LyndonChain}
F=H_0 < H_1 < H_2 < \ldots
\end{eqnarray}
where $H_{i+1}$ is obtained from $H_i$ by extension of all cyclic centralizers. Lyndon showed that $\LyndonGrp$ is
fully residually free \cite{Lyndon60}, hence so are all its subgroups.

In addition to this HNN construction, there are two other constructions of $\LyndonGrp$.  Lyndon's original
construction represented elements as \emph{parametric words}, and Myasnikov, Remeslennikov, and Serbin  \cite{RegularFreeLengthFunctions}
construct $\LyndonGrp$ using
\emph{infinite words}.  The latter construction has proven to be particularly fruitful in solving algorithmic problems,
yielding solutions to the conjugacy and power problems in $\LyndonGrp$.  Two of the important constructions from
\cite{RegularFreeLengthFunctions} that we will need are
normal forms for elements of $\LyndonGrp$ (in terms of infinite words) and a \emph{Lyndon length function}
on $\LyndonGrp$.

A \emph{regular free Lyndon length function} on a group $G$ is a map $l:G\rightarrow A$, where $A$ is an
ordered abelian group, satisfying
\begin{romanenumerate}
\item $\forall\, g\in G: l(g)\geq 0 \; \mathrm{and} \; l(1)=0$,
\item $\forall\, g\in G: \;\; l(g)=l(g^{-1})$,
\item $\forall\, g\in G: \;\; g\neq 1 \implies l(g^2)>l(g)$, and,
\end{romanenumerate}
setting
\[
c_p(g_1,g_2)=\frac{1}{2}\left(l(g_1)+l(g_2)-l(g_1^{-1}g_2)\right),
\]
called the \emph{length of the maximum common prefix},
\begin{romanenumerate}\setcounter{enumi}{3}
\item $\forall\, g_1,g_2\in G:\; c_p(g_1,g_2)\in\integers[t]$,
\item $\forall\, g_1,g_2,g_3\in G: \;\; c_p(g_1,g_2)>c_p(g_1,g_3) \implies c_p(g_1,g_3)=c_p(g_2,g_3)$ \label{L3},
\item $\forall\, g_1,g_2\in G\: \exists\, h, g_1',g_2'\in G$ such that $l(h)=c_p(g_1,g_2)$ and
	$g_1=h\circ g_1'$ and $g_2=h\circ g_2'$ \label{Lyndon:regular}
\end{romanenumerate}
where $\circ$ is defined by
\[
g_1=g_2\circ g_3 \iff g_1=g_2 g_3 \;\;\mbox{and}\;\; l(g_1)=l(g_2)+l(g_3).
\]
For elements $g,h\in G$  we say that $h$ is a \emph{prefix} of $g$ if there exists $g' \in G$ such
that $g=h\circ g'$.

Consider $\integers[t]$ as an ordered abelian group via the
right lexicographic order
induced by the direct sum decomposition
$\integers[t] = \oplus_{m=0}^{\infty} \langle t^m \rangle\simeq \integers^{\infty}$. We use the natural isomorphism
$\integers[t]\simeq \integers^{\infty}$ throughout.
Using the infinite words technique,  \cite{RegularFreeLengthFunctions}  shows that
$\LyndonGrp$ has a regular free Lyndon length
function $l:\LyndonGrp\rightarrow \integers[t]\simeq\integers^{\infty}$.

Recall that any word $w$ over an alphabet $X$ has a \emph{word length} $|w|$ equal to the number of characters in $w$.
$\LyndonGrp$ is generated by $X=X_0\cup \{\t{u}{i}\sst u\in\bigcup_{j=0}^{\infty}R(H_j)\cs i\in\integers\}$, where $X$
generates $F$, so every word $w$ over $X^{\pm}$ has length $|w|$ as a word as well as Lyndon length $l(w)$ as
an element of $\LyndonGrp$.

\begin{example}[A Lyndon length function]\label{Ex:LyndonLength}
Let $F=F(a,b)$ be the free group on generators $a,b$.
We will construct a Lyndon length function $l:G\rightarrow \integers^2$ on the extension of centralizer
$G=\langle a,b,t \sst [ab,t]=1 \rangle$.
For the construction in a more general setting and for proof, refer to \cite{EffectiveJSJ} and \cite{RegularFreeLengthFunctions}.
 Let $w$ be a word over $G$.  First, write $w$ in reduced form
as an element of the HNN-extension,
\[
w=g_1 t^{a_1} g_2 t^{a_2} \cdots g_{m}t^{a_m} g_{m+1}
\]
where $g_i\in F$ for all $i$ and $[g_i,t]\neq 1$ for $i=2,\ldots,m+1$.  Let $l_F$ be the usual length function on $F$
(i.e. $l_F(w)=\min\{|u|\sst u\in\{a^{\pm 1},b^{\pm 1}\}^*\cs u=w\;\mbox{in $F$}\}$)  , and for
$M\in\integers$ set
\[
l_1 (w,M)=l_F (g_1 (ab)^{\epsilon_1 M} g_2 \cdots g_{m}(ab)^{\epsilon_m M} g_{m+1})-m l_F((ab)^M)
\]
where $\epsilon_i=\mathrm{sgn}(a_i)$.
Observe that there exists a positive integer $M_0$ such that for any $M>M_0$,
$l_1(w,M_0)=l_1(w,M)$ (in particular, $M_0=|w|$ will suffice).  Then set the Lyndon length
of $w$ to be
\[
l(w) = \left( l_1(w,M_0),\sum_{i=1}^{m} |a_i|\right).
\]
For example, the word $w=a(ab)^{11}t^{-1}aaba^{-1} t$ (which is in reduced form as written)
has word length $|w|=29$.  For its Lyndon length, use $M=30$ and compute
\[
l_1(w,30)=l_F(a(ab)^{11}(ab)^{-30}aaba^{-1}(ab)^{30})-2(60)=-21.
\]
Hence $w$ has Lyndon length $l(w)=(-21,2)$.

\end{example}

Every $\F$-group $G$ is known to embed into $\LyndonGrp$, and the embedding is effective (\cite{IrredAffineII}).
Since $G$ is finitely generated, $G$
embeds in some finitely generated subgroup $G_n$ of some $H_n$ of (\ref{Eqn:LyndonChain}),
and $G_n$ can be obtained by a sequence of finite extensions of
centralizers,
\begin{eqnarray}\label{Eqn:ChainOfFiniteExtensions}
F=G_0 < G_1 < \ldots < G_n,
\end{eqnarray}
where $G_k < H_k$ for all $k$.  That is, there are finite subsets $R(G_k)\subset R(H_{k})$ and
$T_k=\{\t{u}{i}\sst u\in R(G_k)\cs 1\leq i\leq N_k(u)\}$  such that $G_k$ is the HNN-extension
\begin{eqnarray}\label{Eqn:Gammak}
\langle G_{k-1},T_k \sst \forall\qs u\in R(G_{k-1})\cs 1\leq i,j\leq N_k(u): \; [u,\t{u}{j}]=[\t{u}{i},\t{u}{j}]=1 \rangle.
\end{eqnarray}
Denote by $X_k$ the generating set of $G_k$ such that
$X_0$ is a generating set of $F$ and $X_{k+1}=X_k\cup T_k$.

\section{The compressed word problem in $\F$-groups}
In this section we prove the following theorem.
\begin{theorem}\label{Thm:CWPforFgroups}
Let $G$ be a finitely generated fully residually free group.  Then there is an algorithm that decides the
compressed word problem for $G$ in polynomial time.
\end{theorem}
Since $G$ embeds (effectively) in some $G_n$, it suffices to give a polynomial time
algorithm for the
compressed word problem in $G_n$ (Theorem~\ref{Thm:CWPinGamman}).


\subsection{Normal form}
We will need to represent elements of $G_n$  in a \emph{normal form}, which is based on
the normal form given in
 \cite{RegularFreeLengthFunctions} for infinite word elements of $\LyndonGrp$.

We define normal form in $G_n$ recursively.
For $\alpha=(\alpha_0,\alpha_1,\ldots)\in\integers[t]$ let $\sigma(\alpha)=\sgn(\alpha_d)$ where $d=\deg(\alpha)$.
A word $w$ over $X_0^{\pm}$ is written in normal form if it is
freely reduced.  A word $w$ over
$X_k^{\pm}$ is in normal form if $w$ is written as
\begin{eqnarray}\label{Eqn:NF}
w = g_1 u_{1}^{c_1} \tau_{1}^{\alpha_1} g_2 \ldots g_{m} u_{m}^{c_m} \tau_{m}^{\alpha_m} g_{m+1},
\end{eqnarray}
where $\alpha_i=(\alpha_{i1},\ldots,\alpha_{i N_k(u_i)})\in\integers^{N_k(u_i)}$,
$\tau_{i}^{\alpha_i}=\t{u_i}{1}^{\alpha_{i1}}\t{u_i}{ 2}^{\alpha_{i2}}\cdots \t{u_i}{N_k(u_i)}^{\alpha_{i N_k(u_i)}}$
and
\begin{romanenumerate}
\item for all $i$, $\alpha_i\neq 0$,
\item for each $i$, $g_i$ is a word over $X_{k-1}^{\pm}$,
\item for every $i=1,\dots,m$, either $[u_i,u_{i+1}]\neq 1$ or $[u_i,g_{i+1}]\neq 1$,
\item \label{Eqn:NoCancel} for any integers $q_i\neq 0$ with $\sgn(q_i)=\sigma(\alpha_i)$ we have 
\[
g_1 u_{1}^{q_1} g_2 \ldots g_m u_{m}^{q_m} g_{m+1}
	= g_1 \circ u_{1}^{q_1}\circ g_2\circ \ldots \circ g_m \circ u_{m}^{q_m} \circ g_{m+1}.
\]
\end{romanenumerate}
Note that we do not require the $g_i$ to be written in normal form for $G_{k-1}$.
We call $m$ the number of \emph{syllables} of $w$.

\begin{lemma}\label{Lem:ReducedForm}
For every word $w$ over $X_n^{\pm}$ there is a word $\mathrm{NF}(w)$ in normal form
such that $w=\mathrm{NF}(w)$ in $G_n$ and $|\mathrm{NF}(w)| \leq (10L)^n|w|$, where
$L=\max\{|u|\sst u\in \bigcup_{i=0}^{n} R(G_i)\}$.
\end{lemma}
\begin{proof}

Proceed by induction on $n$.  For $n=0$, $G_0$ is a free group and reduced forms are simply freely-reduced words,
so they exist with $|\mathrm{NF}(w)|\leq |w|$. Assume that the theorem holds for $n-1$.

Using the commutation relations $[u,\t{u}{i}]=[\t{u}{i},\t{u}{j}]=1$ in $G_n$, and an algorithm for the
word problem in $G_{n-1}$ (an algorithm for the conjugacy problem, hence for the word problem,
is given in \cite{RegularFreeLengthFunctions}), we can bring the word
$w$ into the form
\[
w'=h_1 \tau_{1}^{\alpha_1}h_2\ldots h_m\tau_{m}^{\alpha_m}h_{m+1},
\]
where $\tau_{i}^{\alpha_i}$ are as in (\ref{Eqn:NF}) with $\alpha_i\neq 0$ for all $i$,
and for every $i=1,\dots,m$ either $[u_i,u_{i+1}]\neq 1$ or $[u_i,h_{i+1}]\neq 1$.
Notice that $|w'|\leq |w|$.

To produce a reduced form from $w'$, we appeal to \cite{RegularFreeLengthFunctions}, which constructs normal forms
for elements of $\LyndonGrp$, but without proof of the length bound that we require.
Only minor changes to that construction are needed, and we draw the reader's attention to the relevant sections.

The key fact is the following: for any word $g$ over $X_{n-1}^{\pm}$ and any $u\in R(G_{n-1})$ we have that,
for any $r> (10L)^{n-1}|g|$,
\begin{eqnarray}\label{Eqn:keyFact}
u^{r+1} g = u\circ (u^{r}g) & \mbox{and} & g u^{r+1}=(gu^r)\circ u.
\end{eqnarray}
The proof of this fact is part of Lemma~7.1 of \cite{RegularFreeLengthFunctions}, which shows that the above holds as long as $r$ is greater than
the number of syllables in a normal form of $g$.  Since $g\in G_{n-1}$, we have by induction that
$|\mathrm{NF}(g)|\leq (10L)^{n-1}|g|$ hence $\mathrm{NF}(g)$ has at most $(10L)^{n-1}|g|$ syllables.

There is an isomorphism $\phi$ from our HNN-representation of $\LyndonGrp$ to the infinite words representation.  The
word $w'$ corresponds, via $\phi$, to what in \cite{RegularFreeLengthFunctions} is called a
\emph{reduced R-form}.  Lemma~6.13 of \cite{RegularFreeLengthFunctions} constructs normal forms
from reduced R-forms, and the first step of this construction produces
a form that corresponds, via $\phi$, to our
normal form.  The construction attaches powers of $u_{i-1}$ and $u_i$ to $h_i$,
using rewritings of the form
\begin{eqnarray*}\label{Eqn:transformation}
h_{i} \tau_{i}^{\alpha_{i}} & \longrightarrow &
    (h_i u_i^{\sigma(\alpha_i)r_i})(u_i^{-\sigma(\alpha_i)}\tau_i^{\alpha_i}), \\
 \tau_{i}^{\alpha_{i}}h_{i+1} & \longrightarrow &
   (\tau_i^{\alpha_i} u_i^{-\sigma(\alpha_i)})(u_i^{\sigma(\alpha_i)r_{i+1}}h_{i+1}),
\end{eqnarray*}
where $r_i=(10L)^{n-1}|h_i|+1$,  with property (\ref{Eqn:keyFact}) being used to achieve condition 
(iv).
It produces a normal form
\[
\mathrm{NF}(w')=g_1 u_{1}^{c_1} \tau_{1}^{\alpha_1} g_2 \ldots g_{m} u_{m}^{c_m} \tau_{m}^{\alpha_m} g_{m+1},
\]
where $|g_i|\leq r_i|u_{i-1}|+|h_i|+r_i|u_i|$ and $|c_i|\leq r_i+r_{i+1}$ for all $i$.  Then the length of $\mathrm{NF}(w')$
has the bound
\begin{eqnarray*}
|\mathrm{NF}(w')|& =&\sum_{i=1}^m \left(|\tau_i^{\alpha_i}| +|c_i| |u_i| + |g_i| \right)+|g_{m+1}| \\
    & \leq & \left(|w'| -\sum_{i=1}^{m+1} |h_i|\right) + \sum_{i=1}^m \left((r_i+r_{i+1})L+2r_i L+|h_i|\right)+2r_{m+1}L+|h_{m+1}|\\
    & \leq & |w| + 4L\sum_{i=1}^{m+1}r_i \leq |w'| +4L(10^{n-1}L^{n-1}|w'|+|w'|)\\
    & \leq & (10L)^n |w|
\end{eqnarray*}
as required.
\end{proof}

\begin{example}[Normal forms]
Consider again the word $w=a(ab)^{11}t^{-1}aaba^{-1}t$ from Example~\ref{Ex:LyndonLength}.
A normal form for $w$ is given by
\[
a \left((ab)^{12}\right)t^{-1} (b^{-1}a^{-1}aaba^{-1}ab) \left((ab)^{-1}\right) t
\]
where $g_1=a$, $c_1=12$, $g_2=b^{-1}a^{-1}aaba^{-1}ab$, $c_2=-1$.
It is not necessray to freely reduce $g_2$, though we may do so if desired.
Notice that for any $q_1<0$ and $q_2>0$,
\[
a (ab)^{q_1} (b^{-1}a^{-1}aaba^{-1}ab)  (ab)^{q_2}
	 = a \circ (ab)^{q_1}\circ (b^{-1}a^{-1}aaba^{-1}ab) \circ (ab)^{q_2} ,
\]
satisfying (iv).
\end{example}


\subsection{Algorithm for the compressed word problem}
To solve the compressed word problem in $G_n$, we construct a reduction of the word problem in $G_n$
to the word problem in $F$, then apply the
reduction to compressed words and use Lemma~\ref{Lem:CWPinF} to solve the compressed word problem in $F$.

\begin{definition}
For $P\in\naturals$,
define a homomorphism $\varphi_{(n,P)}: G_n \rightarrow G_{n-1}$ by
setting $\varphi_{(n,P)}$ to be the identity on $G_{n-1}$ and setting
$\varphi_{(n,P)} (\t{u}{i}) = u^{P^i}$.
\end{definition}
Note that $\varphi_{(n,P)}$ is a homomorphism since, for every $i,j$,
\[
[u,\varphi_{(n,P)} (\t{u}{i})]=[u,u^{P^i}]=1=[u^{P^i},u^{P^j}]=[\varphi_{(n,P)}(\t{u}{i}),\varphi_{(n,P)}(\t{u}{j})].
\]

Let $w$ be a word over $X_n^{\pm}$.
Recalling from (\ref{Eqn:Gammak}) that $N_k(u)$ is the number of letters $\t{u}{i}$ for a given $u\in R(G_k)$,
set $N=1+\max\{N_k(u)\sst k\in [0,n-1]\cs u\in \bigcup_{i=0}^{n-1}R(G_i)\}$.
For $P\in\naturals$ define a sequence of $n$ constants $P_n,P_{n-1},\ldots,P_1$ by $P_n=P$ and
\[
P_{i-1} = P_i^N \cdot L,
\]
i.e. $P_{n-i}=P^{N^{i}}L^{N^{i-1}}L^{N^{i-2}}\cdots L$,
and define a homomorphism
$\Phi_{(n,P_n)}:G_n\rightarrow F$ by  the composition
$\Phi_{(n,P)} = \varphi_{(1,P_1)} \varphi_{(2,P_2)} \cdots \varphi_{(n,P_n)}$.
The sequence is defined so that when
$P_n>(10L)^n |w|$,  $P_{i-1}$ is an upper bound on the length of
$\varphi_{(i,P_i)}\cdots\varphi_{(n,P_n)}(w)$, as we will see below.

\begin{theorem}\label{Thm:BPreplacement}
Let $G_n$ be obtained by a sequence of extensions of centralizers as in (\ref{Eqn:ChainOfFiniteExtensions})
and let $w$ be a word over $X_n^{\pm}$.
Then for any $P>(10L)^n |w|$,
\[
\Phi_{(n,P_n)}(w) = 1 \mbox{ in $F$} \iff w=1 \mbox{ in $G_n$}.
\]
\end{theorem}
\begin{proof}
Since $\Phi_{(n,P_n)}$ is a homomorphism, if $w=1$ in $G_n$ then $\Phi_{(n,P_n)}(w)=1$ in $F$. It remains to show
that for any $P>(10L)^n |w|$,
\[
w\neq 1 \mbox{ in $G_n$} \implies \Phi_{(n,P)}(w)\neq 1 \mbox{ in $F$}.
\]
We proceed by induction on $n$.
Letting $\Phi_{(0,P_0)} :F\rightarrow F$ be the identity map, there is nothing to prove in the  base case $n=0$.
Assume the theorem holds up to $n-1$ and that $w\neq 1$ in $G_n$.
If  for all $\t{u}{i}\in T_n$ and
$\epsilon=\in\{\pm 1\}$ the letter
$\t{u}{i}^{\epsilon}$ does not appear in $w$, then $w\in G_{n-1}$ so
$\Phi_{(n,P_n)}(w)=\Phi_{(n-1,P_{n-1})}(w)$.
Since $P_{n-1}>P_n>(10L)^{n-1} |w|$ the induction assumption applies,
so $\Phi_{(n,P_n)}(w)\neq 1$ in $F$.

Now assume that $\t{u}{i}^{\epsilon}$ appears in $w$ for at least one $\t{u}{i}\in T_n$. Let
\[
\mathrm{NF}(w) = g_1 u_{1}^{c_1} \tau_{1}^{\alpha_1} g_2 \ldots g_{m} u_{m}^{c_m} \tau_{m}^{\alpha_m} g_{m+1}
\]
be a normal form of $w$, as in Lemma~\ref{Lem:ReducedForm}.
Since $\t{u}{i}^{\epsilon}$ appears we have $m\geq 1$.
We claim that $\varphi_{(n,P_n)}(u_{i}^{c_i}\tau_{i}^{\alpha_i})$ is a non-zero power of $u_{i}$
of sign $\sigma(\alpha_i)$.
We simplify notation by setting $u=u_i$, $a=\alpha_i$, and $d=N_{n-1}(u)$.
We have
\[
\varphi_{(n,P_n)}(\tau_{i}^{\alpha_i}) =
	\varphi_{(n,P_n)}(\t{u}{1}^{a_1}\cdots \t{u}{d}^{a_{d}}) =
	u^{ a_{d}P_n^{d}+a_{d-1}P_n^{d-1}+\ldots +a_1 P_n }
\]
and we want a lower bound of the magnitude of the exponent of $u$.
Since, for all $s$,
\[
|a_s|\leq |\mathrm{NF}(w)|\leq (10L)^n |w| \leq P_n-1,
\]
we have that
\[
\sum_{s=1}^{d-1} |a_s|P_n^s \leq \sum_{s=1}^{d-1}(P_n-1)P_n^s = P_n^{d}-P_n.
\]
Hence $|a_d P_n^d|-|a_{d-1}P_n^{d-1}+\ldots+a_1 P_n|\geq P_n$, and so
\[
a_{d}P_n^{d}+a_{d-1}P_n^{d-1}+\ldots +a_1 P_n = C_i
\]
where $|C_i|\geq P_n$ and $\mathrm{sgn}(C_i)=\mathrm{sgn}(a_d)=\sigma(a)$. Then
\[
\varphi_{(n,P_n)}(u_{i}^{c_i}\tau_{i}^{\alpha_i}) = u^{C_i+c_i}
\]
with $C_i+c_i\neq 0$ (since $|c_i|\leq |\mathrm{NF}(w)<P_n$) and $\mathrm{sgn}(C_i+c_i)=\sigma(\alpha_i)$,
proving the claim.

Since $\varphi_{(n,P_n)}$ is
the identity on $G_{n-1}$, we have, using property (iv) of normal forms,
\[
\varphi_{(n,P_n)}(w)=\varphi_{(n,P_n)}(\mathrm{NF}(w))
    = g_1 \circ u_1^{C_1+c_1}\circ g_2 \circ \cdots \circ g_{m}\circ u_m^{C_m+c_m}\circ g_{m+1}.
\]
In particular, $l(\varphi_{(n,P_n)}(w))\geq l(u_1^{C_1+c_1})>0$
hence $\varphi_{(n,P_n)}(w)\neq 1$ in $G_{n-1}$. We have
$\Phi_{(n,P_n)}(w)=\Phi_{(n-1,P_{n-1})}(\varphi_{(n,P)} (w))$ and we can apply the induction hypothesis
to $\varphi_{(n,P)} (w)$
since $P_{n-1}$ is large enough.  Indeed, in the worst case $w=\t{u}{i}^{|w|}$ where
$|u|=L$ and $i=N-1$ making
\[
|\varphi_{(n,P_n)}(w)|= |u^{P_n^{N-1}|w|}| = |w|P_n^{N-1}L<P_n^{N}L=P_{n-1},
\]
so by induction $1\neq\Phi_{(n-1,P_{n-1})}(\varphi_{(n,P_n)}(w))=\Phi_{(n,P_n)}(w)$ in $F$.
\end{proof}

We now can solve the word problem in $G_n$ by
setting $P=(5L)^n|w|+1$ and checking if $\Phi_{(n,P_n)}(w)$ is
trivial in $F$.
Notice that the bound on the length of $\Phi_{(n,P_n)}(w)$ is given by
\[
P_0= P^{N^{n}}L^{N^{n-1}}L^{N^{n-2}}\cdots L.
\]
We use this reduction to solve the \emph{compressed} word problem in $G_n$.
\begin{theorem}\label{Thm:CWPinGamman}
Let $G_n$ be a group obtained from a free group
by a finite sequence of finite extensions of centralizers as in
(\ref{Eqn:ChainOfFiniteExtensions}).
There is an algorithm that decides the
compressed word problem for $G_n$ in polynomial time.
\end{theorem}
\begin{proof}
Let $\A$ be a compressed word over $X_n^{\pm}$.  For any word $w$ and any $q\in\integers$ we can write a
straight-line program $W^q$ of
size $2|w|+\log_2|q|$ producing
$w^q$.  Indeed, the root production is $W^q\rightarrow W^{q/2} W^{q/2}$, where $W^{q/2}$ produces $w^{q/2}$, and
we continue by induction (make the appropriate changes when $q$ is odd), noting that we get at most
$\log_2 |q|$ non-terminals
of the form $W^p$.
We can obtain the program $W^1$, which produces $w$ and has size $2|w|$,
by successively dividing $w$ in half.
Consequently, for each
$u\in R(G_n)$ and $q\in\integers$, we have an SLP with root
$U^q$ producing $u^q$ and having size $2|u|+\log_2 |q|$.

Set $P=(10L)^n|w_{\A}|+1$ and build an SLP $\A_n$ by replacing every production of $\A$ of the
form
\[
A\rightarrow \t{u}{i}^{\epsilon},
\]
 where $\t{u}{i}\in T_n$ and $\epsilon=\pm 1$,
by
\[
A\rightarrow U^{\epsilon P_n^i}.
\]
Notice that $w_{\A_n}=\varphi_{(n,P_n)}(w_{\A})$.  Repeat this replacement process
for $\A_n$ to produce $\A_{n-1}$ and continue until we get $\A_1$, which is an SLP producing $\Phi_{(n,P)}(w_{\A})$.
By Theorem~\ref{Thm:BPreplacement}, $w_{\A_1}=1$ in $F$ if and only if $w_{\A}=1$ in $G_n$
so we now apply Lohrey's
algorithm (Lemma~\ref{Lem:CWPinF}) to decide if $w_{\A_1}=1$ in $F$.

We need to show that the size of $\A_1$ is polynomial (in fact, linear) in the size of $\A$.  At each level $k$,
we add, for each $u\in R(G_k)$,
programs $U^{P^1},U^{P^2},\ldots,U^{P^{N_k(u)}}$.  Recalling that
$N=1+\max\{N_k(u)\sst k\in [0,n-1]\cs u\in \bigcup_{i=0}^{n-1}R(G_i)\}$, each new $U^{P^i}$
adds less than
\[
2|u|+\log_2 |P^i|\leq 2L+\log_2(P_k^N)
\]
new non-terminals to $\A_k$. Letting $M=\mathrm{max}_k\{|R(G_k)|\}$, level $k$ introduces less than
\[
2LM + NM\log_2(P_k)
\]
new non-terminals.  In total, over all $n$ levels, the number of new non-terminals is bounded by
\[
2nLM + NM\sum_{i=0}^{n-1} \log_2 (P_{n-i}).
\]
Noting that $L,M,n$ are constants (i.e. they depend of $G_n$, not on $w$) and recalling
$P_{n-i}=P^{N^{i}}L^{N^{i-1}}L^{N^{i-2}}\cdots L$,
we have
that the number of new non-terminals is in
\begin{eqnarray*}
O\left(\sum_{i=0}^{n-1}\log(P_{n-i})\right)&=&O\left(\sum_{i=0}^{n-1}N^i\log(P)\right)
= O(\log(P))\\
&=& O\left(\log((10L)^n 2^{|\A|}+1)\right) = O(|\A|).
\end{eqnarray*}
Therefore $|\A_1|\in O(|\A|)$ and since Lohrey's algorithm runs in polynomial time in $|\A_1|$
we have a polynomial time algorithm for the compressed word problem in $G_n$.
\end{proof}

\section{Word problem in the automorhpism group of an $\F$-group}
In \cite{Schleimer08}, Schleimer uses a polynomial time algorithm for the compressed word problem in a free group
to produce a polynomial time algorithm for the word problem in its automorphism group.  We apply the same method
to $\F$-groups.

\begin{theorem}\label{Thm:WPinAut}
Let $G$ be a finitely generated fully residually free group.  Then the word problem for $\Aut(G)$ is decidable in polynomial time.
\end{theorem}

The theorem follows from Theorem~\ref{Thm:CWPforFgroups} and known results,
which we collect and summarize here.
The main idea is that the word problem in $\Aut(G)$ reduces to the compressed word problem in $G$:
\begin{lemma}[Proposition 2 of \cite{LohreySchleimer07}]
Let $G$ be a finitely generated group and $H$ a finitely generated subgroup of $\Aut(G)$.  Then the
word problem in $H$ reduces in logarithmic space to the compressed word problem in $G$.
\end{lemma}
To construct the reduction, one needs the generators of $H$ to be described by their action on generators of $G$.
That is, if $G=\langle g_1,\ldots,g_n\rangle$ then each $\phi_i\in H$ must be given by
\begin{equation} \label{Eqn:AutsGivenByWords}
\phi_i(g_j)=w_{ij}(g_1,\ldots,g_n),
\end{equation}
where $w_{ij}(g_1,\ldots,g_n)$ is a word
over the alphabet $\{g_1\ldots,g_n\}^{\pm 1}$. Now suppose $H=\langle \phi_1,\ldots,\phi_k\rangle$
and we want to decide if a word $\phi_{i_1}\ldots\phi_{i_m}$ represents the trivial element of $H$. Build a set of
non-terminals $\{A_{j,p},\overline{A_{j,p}}\}$, where $j\in\{1,\ldots,n\}$ and $p\in\{1,\ldots,m\}$, with productions
\begin{eqnarray*}
A_{j,0} & \rightarrow & g_j,\\
\overline{A_{j,0}} & \rightarrow & g_j^{-1},\\
A_{j,p} & \rightarrow & w_{i_p j}(A_{1,p-1},\ldots,A_{n,p-1}),\; p\geq 1, \\
\overline{A_{j,p}} & \rightarrow & (w_{i_p j}(A_{1,p-1},\ldots,A_{n,p-1}))^{-1},\; p\geq 1,
\end{eqnarray*}
where $w_{i_p j}(A_{1,p-1},\ldots,A_{n,p-1})$ is the word $w_{i_p j}$ with every instance of $g_i$
replaced by $A_{i,p-1}$ and of $g_i^{-1}$ by $\overline{A_{i,p-1}}$. One sees that
$w_{A_{j,m}}=\phi_{i_1}\ldots\phi_{i_m} (g_j)$. Then the word problem in $H$ reduces to checking that
$w_{A_{j,m}}=g_j$ for all $j$, i.e. it reduces to $n$ instances of the compressed word problem in $G$.

To prove Theorem~\ref{Thm:WPinAut} then, it suffices to show that $\Aut(G)$ is finitely generated and that every generator
can be described as in (\ref{Eqn:AutsGivenByWords}).

First, consider the case when $G$ is freely indecomposable.  The structure of the automorphism group of such $G$ has been
described in \cite{BumaginIso} using an Abelian JSJ-decomposition of $G$.  It follows from the results in \S 5 of that paper
that $\Aut(G)$ is finitely generated and the automorphisms can be described as in (\ref{Eqn:AutsGivenByWords}).
Note that constructing an Abelian JSJ-decomposition of an $\F$-group is effective (Theorem~13.1 of \cite{EffectiveJSJ}).

For the general case, let $G$ be any \F-group. Then $G$ has a \emph{Grushko decomposition} as
a free product
\[
G = G_1 * \cdots * G_k * F_r,
\]
where the $G_i$ are freely indecomposable non-cyclic groups and $F_r$ is a free group of rank $r$.
This decomposition is unique in the sense that
any other such decomposition has the same $k$ and $r$ and its freely indecompasable non-cyclic factors are
conjugated in $G$ to the factors $G_1,\ldots,G_k$.  One can effectively find a Grushko decomposition for
\F-groups \cite{EffectiveJSJ}.

The automorphism group of a free product has been described by
Fouxe-Rabinovitch and Gilbert \cite{Gilbert85} in terms of the automorphisms of its factors.
$\Aut(G)$ is generated
by the following automorphisms.
\begin{romanenumerate}
\item \emph{Permutation automorphisms.}  For each pair of isomorphic factors $G_i\simeq G_j$, fix an automorhism $\phi_{ij}$.
Choose $\phi_{ij}$ such that the collection is compatible, that is if $G_i\simeq G_j$ and $G_j\simeq G_k$ then
$\phi_{ik}=\phi_{jk}\phi_{ij}$.
\item \emph{Factor automorphisms.} Each automorphism of $G_i$ and of $F_r$ induces an automorphism of $G$ by acting
as the identity on all other factors. Any product of
such automorphisms is called a factor automorphism.
\item \emph{Whitehead automorhpisms.} Let $S$ be a basis of $F_r$.  An automorhpism of $G$ is a Whitehead automorphism if
there is an $x$ in some $G_i$ or in $S$ such that each factor $G_j$ is conjugated by $x$ of fixed pointwise,
and each $s\in S$ is sent to one of $s,sx,x^{-1}s,x^{-1}sx$.
\end{romanenumerate}
It follows from Theorem~4.13 of \cite{BumaginIso} that we can construct a compatible set of permutation automorphisms.
Since each $G_i$ is freely indecomposable we can construct a finite generating set for $\Aut(G_i)$. The automorphism group
of a free group $F(x_1,\ldots,x_r)$ is well-known to be finitely generated by the Nielsen automorphisms,
\begin{eqnarray*}
\alpha_i (x_k) & = & \braced{x_k^{-1}}{k=i}{x_k}{k\neq i}, \;i\in\{1,\ldots,r\} \\
\beta_{ij} (x_k) & = & \braced{x_k x_j}{k=i}{x_k}{k\neq i}, \;i,j\in\{1,\ldots,r\},i\neq j.
\end{eqnarray*}
Consequently, the factor automorphisms are finitely generated.
Since each $G_i$ (and $F_r$) is finitely generated, the set of Whitehead automorphisms is finitely generated.  Therefore we have proven
the following lemma, which completes the proof of Theorem~\ref{Thm:WPinAut}:
\begin{lemma}
Let $G$ be an \F-group.  Then $\Aut(G)$ is finitely generated and one can construct a generating set
in the form (\ref{Eqn:AutsGivenByWords}).
\end{lemma}

\bibliographystyle{alpha}
\bibliography{compressed_words}
\end{document}